\documentclass{amsart}
\usepackage{amssymb}
\usepackage{amsmath}
\usepackage{amsthm}
\usepackage{verbatim}
\usepackage{dsfont}

\newtheorem{theorem}{Theorem}[section]

\newtheorem{lemma}[theorem]{Lemma}

\theoremstyle{definition}

\theoremstyle{remark}

\numberwithin{equation}{section}

\newcommand{\enproof}{\hspace*{\stretch{1}}\qedsymbol}  

\newcommand{\norm}[1]{\left\|#1\right\|}  

\newcommand{\Measures}{\mathrm{M}}
\newcommand{\FSa}{\mathrm{B}}
\newcommand{\Fa}{\mathrm{A}}
\newcommand{\operators}{\mathcal{B}}
\newcommand{\VN}{\mathrm{VN}}

\newcommand{\dd}{\,\mathrm{d}}

\newcommand{\ltwo}{\ell^2}

\newcommand{\set}[1]{\left\{#1\right\}}  

\newcommand{\complexs}{\mathds{C}}    

\newcommand{\integers}{\mathds{Z}}    

\newcommand{\Kcal}{\mathcal{K}}

\newcommand{\abs}[1]{\left|#1\right|}

\begin{document}

\title{Idempotents with small norms}

\author{Jayden Mudge}
\email{jayden.mudge@msor.vuw.ac.nz}
\address{School of Mathematics and Statistics, Victoria University of Wellington, Wellington, New Zealand}
\thanks{The first author is supported by the Victoria University of Wellington Masters
by thesis Scholarship 2015.}

\author{Hung Le Pham}
\email{hung.pham@vuw.ac.nz}
\address{School of Mathematics and Statistics, Victoria University of Wellington, Wellington, New Zealand}


\begin{abstract}
Let $\Gamma$ be a locally compact group. We answer two questions left open in \cite{saeki2} and \cite{stan}: 
\begin{enumerate}
	\item For abelian $\Gamma$, we prove that if $\chi_S \in \FSa(\Gamma)$ is an idempotent with norm $\norm{\chi_S} < \frac{4}{3}$, then $S$ is the union of two cosets of an open subgroup of $\Gamma$. 
	\item For general $\Gamma$, we prove that if $\chi_S \in M_{cb}\Fa(\Gamma)$ is an idempotent with norm $\norm{\chi_S}_{cb} < \frac{1 + \sqrt{2}}{2}$, then $S$ is an open coset in $\Gamma$.
\end{enumerate}
\end{abstract}

\maketitle

\section{Introduction}

\noindent In his 1968 papers, Saeki determined idempotent measures on a locally compact abelian group $G$ with small norms. These are equivalent to determining idempotent functions in the Fourier--Stieltjes algebras $\FSa(\Gamma)$ on a locally compact abelian group $\Gamma$ with small norms (where $\Gamma$ and $G$ could be taken as Pontryagin duals of each other). The statements of Saeki's results in the Fourier--Stieltjes setting are:

\begin{theorem}[Saeki]
\label{Saeki}
Let $\Gamma$ be a locally compact abelian group, and let $\varphi$ be an idempotent function in $\FSa(\Gamma)$ so that $\varphi=\chi_S$ for some nonempty $S\subseteq \Gamma$. Then
\begin{enumerate}
	\item {\rm (}\cite{saeki1}{\rm )} If $\norm{\varphi}<\frac{1 + \sqrt{2}}{2}$, then $S$ is an open coset of $\Gamma$.
	\item {\rm (}\cite{saeki2}{\rm )} If $\norm{\varphi}\in (1, \frac{\sqrt{17} + 1}{4})$, then $S$ is the union of two cosets of an open subgroup of $\Gamma$ but is not a coset itself.
\end{enumerate}
\end{theorem}

For abelian $\Gamma$, it is well-known (see \cite[page 73]{Rudin}) that if $S$ is an open coset of $\Gamma$, then $\norm{\chi_S}=1$, and whereas if $S$ is the union of two cosets of an open subgroup of $\Gamma$ but is not a coset itself, then
\begin{align}\label{list of norms}
	\norm{\chi_S}=\begin{cases}
	\frac{2}{q\sin (\pi/2q)}&\ \text{if $q$ is odd}\\
	\frac{2}{q\tan (\pi/2q)}&\ \text{if $q$ is even}\\
	\frac{4}{\pi}&\ \text{if $q=\infty$}
	\end{cases}
\end{align}
where $q$ is the ``relative order'' of the two cosets forming $S$. The largest value in \eqref{list of norms} is $4/3$ when $q=3$ and the smallest one is $\frac{1 + \sqrt{2}}{2}$ when $q=4$. In particular, the number $\frac{1 + \sqrt{2}}{2}$ in Theorem \ref{Saeki} (i) is sharp. 

The paper \cite{saeki2} asked whether or not the interval $(1,\frac{\sqrt{17} + 1}{4})$ in Theorem \ref{Saeki} (ii) could be increased to $(1,\frac{4}{3})$, and we answer this in affirmative in Theorem \ref{stronger version of Saeki}. Note that the interval $(1,\frac{4}{3})$ is sharp because of the discussion in the previous paragraph and also since there are idempotents $\chi_S$ of $\FSa(\Gamma)$ with $\norm{\chi_S}=4/3$ but $S$ is not any union of two cosets of open subgroup of $\Gamma$  (see the last paragraph of \cite{saeki2}).  

Lesser is known about idempotents in $\FSa(\Gamma)$ with small norms for general locally compact group $\Gamma$. Ilie and Spronk \cite{IlieSpronk} proved that $\chi_S$ is an idempotent in $\FSa(\Gamma)$ with $\norm{\chi_S}=1$ if and only if $S$ is an open coset of $\Gamma$. More generally, Stan proved the following.

\begin{theorem}[Stan \cite{stan}]
\label{Stan}
Let $\Gamma$ be a locally compact group, and let $\varphi$ be an idempotent function in $M_{cb}\Fa(\Gamma)$ so that $\varphi=\chi_S$ for some nonempty $S\subseteq \Gamma$. If $\norm{\varphi}_{cb}<\frac{2}{\sqrt{3}}$, then $S$ is an open coset of $\Gamma$, and in which case $\norm{\varphi}_{cb}=1$. 
\end{theorem}

Here $M_{cb}\Fa(\Gamma)$ is the completely bounded multiplier algebra $M_{cb}\Fa(\Gamma)$ of the Fourier algebra $\Fa(\Gamma)$ and is defined as follows. Since the Fourier algebra $\Fa(\Gamma)$ is the predual of the group von Neumann algebra $\VN(\Gamma)$, it has the canonical operator space structure, which makes it a completely contractive operator algebra (see the monograph \cite{EffrosRuan} for more details). The completely bounded multiplier algebra $M_{cb}\Fa(\Gamma)$ of $\Fa(\Gamma)$ consists of those continuous functions $\varphi:\Gamma\to\complexs$ such that the mapping
\[
	f\mapsto \varphi\cdot f, \Fa(\Gamma)\to \Fa(\Gamma)\,,
\]
is completely bounded, and its completely bounded norm is denoted as $\norm{\varphi}_{cb}$ (whereas the Fourier--Stieltjes norm on $\Gamma$ will be simply denoted as $\norm{\cdot}$ in this paper). In general, we have 
\[
	\FSa(\Gamma)\subseteq M_{cb}\Fa(\Gamma)\,,\quad\text{with a decreasing of norms}\,,
\]
but, for amenable locally compact groups $\Gamma$,
\[
	\FSa(\Gamma)= M_{cb}\Fa(\Gamma)\,,\quad\text{isometrically}\,.
\]
Thus an idempotent of $\FSa(\Gamma)$ with a small norm is always an idempotent of $M_{cb}\Fa(\Gamma)$ with a small(er) norm.

In Theorem \ref{stronger version of Stan}, we increase the number $\frac{2}{\sqrt{3}}$ in Stan's Theorem \ref{Stan} to the sharp bound of $\frac{1 + \sqrt{2}}{2}$, and so obtaining a generalisation of the first mentioned result of Saeki, Theorem \ref{Saeki} (i), to general locally compact groups.

\section{Idempotents of $M_{cb}\Fa(\Gamma)$ with norm lesser than $\frac{1 + \sqrt{2}}{2}$}

In this section, let $\Gamma$ be any locally compact group, and let $\chi_S$ be an idempotent of $M_{cb}\Fa(\Gamma)$ with $\norm{\chi_S}_{cb}<\frac{1 + \sqrt{2}}{2}$. Our aim is to show that $S$ is an open coset of $\Gamma$. It is obvious that $S$ is open, and so, it remains to show that $S$ is a coset of $\Gamma$. By \cite[Corollary 6.3 (i)]{Spronk}, it is sufficient for us to consider the case where $\Gamma$ is discrete.

We first make a simple observation.

\begin{lemma}\label{progression}
For any $s \in S$ and $t \in \Gamma$, if $st \in S$ {\rm (}resp. $ts \in S${\rm )}, then $st^n \in S$ for every $n\in\integers$ {\rm (}resp. $t^ns \in S$ for every $n\in\integers${\rm )}.
\end{lemma}
\begin{proof}
By translation, we may (and shall) suppose that $s=e$, the identity of $\Gamma$. Consider $\Gamma_0$ be the (abelian) group generated by $t$, then 
\[
	\norm{\chi_{S\cap\Gamma_0}}=\norm{\chi_{S\cap\Gamma_0}}_{cb}\le \norm{\chi_S}_{cb}<\frac{1 + \sqrt{2}}{2}.
\]
So by Saeki's Theorem \ref{Saeki} (i), we see that $S\cap\Gamma_0=\Gamma_0$. This gives the lemma.
\end{proof}

To get more information out of the assumption on $\norm{\chi_S}_{cb}$, we shall follow in the footsteps of \cite{stan} and use the connection shown in \cite{BozejkoFendler} between the norm $\norm{\cdot}_{cb}$ of $M_{cb}\Fa(\Gamma)$ and the Schur multiplier norms described below.

Denote by $\Kcal_0$ the space of matrices that have only finitely many nonzero entries whose rows and columns are indexed by $\Gamma$. Then $\Kcal_0$ is identified with a subspace of $\operators(\ltwo(\Gamma))$. Recall that the \emph{Schur multiplication} of two matrices $A$ and $X$, indexed by $\Gamma$, is defined as 
\[
	(A\bullet X)(s,t):=A(s,t)X(s,t)\quad(s,t\in\Gamma)\,,
\]
and for each matrix $A$, indexed by $\Gamma$, its \emph{Schur multiplier norm} is 
\begin{align*}
	\norm{A}_{\text{Schur}}:=\sup\set{\frac{\norm{A\bullet X}_{\operators(\ltwo(\Gamma))}}{\norm{X}_{\operators(\ltwo(\Gamma))}}\colon X\in \Kcal_0}\,.
\end{align*}
Of course, this discussion works for any index set $\Gamma$, and a particular matrix that is useful for us is the following $3\times 3$ matrix
\begin{align}\label{forbidden matrix}
F_0:=\begin{pmatrix}
1 & 1 & 1\\
1 & 1 & 0\\
1 & 0 & 1
\end{pmatrix}\,.
\end{align}
Using the orthogonal matrix
$U:=\frac{1}{2}\begin{pmatrix}
0 & \sqrt{2} & \sqrt{2}\\
\sqrt{2} & 1 & -1\\
\sqrt{2} & -1 & 1
\end{pmatrix}$ and the vector $\xi:=\frac{1}{2}\begin{pmatrix}
\sqrt{2}\\
1\\
1
\end{pmatrix}$, we see that 
\[
	\norm{F_0}_{\text{Schur}}\ge\norm{A\bullet U}_{\operators(\ltwo)}\ge\frac{\norm{(A\bullet U)\xi}_{\ltwo}}{\norm{\xi}_{\ltwo}}=\frac{\sqrt{26}}{4}>\frac{1+\sqrt{2}}{2}.
\]
As a matter of fact, it is proved in \cite[Proposition 5.1(8)]{levene} that $\norm{F_0}_{\text{Schur}}=\frac{9}{7}$, but the above simple calculation is sufficient for our purpose. Hence, any matrix $A$ that has a submatrix of the form $F_0$ in \eqref{forbidden matrix} must satisfy
\[
	\norm{A}_{\text{Schur}}>\frac{1+\sqrt{2}}{2}\,. 
\]

Returning to our problem on the group $\Gamma$, each function $\varphi:\Gamma\to\complexs$ defines a matrix $M_\varphi$, indexed by $\Gamma$, by setting
\[
	M_\varphi(s,t):=\varphi(s^{-1}t)\quad(s,t\in \Gamma)\,.
\]
An important fact shown in \cite{BozejkoFendler} is that  
\begin{align}\label{cb and Schur norms}
	\norm{\varphi}_{cb}=\norm{M_\varphi}_{\text{Schur}}\,.
\end{align}
Hence, the previous paragraph implies that $M_{\chi_S}$ cannot have \eqref{forbidden matrix} as a submatrix.

Our main result of this section is the following.

\begin{theorem}
\label{stronger version of Stan}
Let $\Gamma$ be a locally compact group, and let $\varphi$ be an idempotent function in $M_{cb}\Fa(\Gamma)$ so that $\varphi=\chi_S$ for some nonempty $S\subseteq \Gamma$. If $\norm{\varphi}_{cb}<\frac{1+\sqrt{2}}{2}$, then $S$ is an open coset of $\Gamma$. 
\end{theorem}
\begin{proof}
As discussed above, we may (and shall) suppose that $\Gamma$ is discrete. Also, applying a translation if necessary, we suppose that $e\in S$. So it remains to prove that $S$ is a subgroup of $\Gamma$. 

By Lemma \ref{progression}, we see that if $u\in S$, then $u^{n}\in S$ for every $n\in\integers$. Thus it remains to show that $S$ is closed under multiplication.

We next \emph{claim} that if $u, v\in S$, then either $uv\in S$ or $vu\in S$. Indeed, assume towards a contradiction that both $uv\notin S$ and $vu\notin S$. Then the submatrix of $M_{\chi_S}$ with rows $e, u^{-1}, v^{-1}$ and colums $e, u, v$ is
\begin{align*}
\begin{pmatrix}
\chi_S(e) & \chi_S(u) & \chi_S(v)\\
\chi_S(s) & \chi_S(u^2) & \chi_S(uv)\\
\chi_S(v) & \chi_S(vu) & \chi_S(v^2)
\end{pmatrix}=\begin{pmatrix}
1 & 1 & 1\\
1 & 1 & 0\\
1 & 0 & 1
\end{pmatrix}
\end{align*}
by the previous paragraph. This contradicts the previous discussion.

Finally, suppose that $u, v\in S$, the proof is completed if we can show that $uv\in S$. The claim shows that either $uv\in S$ or $vu\in S$. Assume the latter holds, then from Lemma \ref{progression} with $s=v$ and $t=u$, we obtain that $vu^{-1}\in S$. Since we must have $u^{-1}\in S$, this in turn implies, by a similar argument, that $v^{-1}u^{-1}\in S$. But then, since $v^{-1}u^{-1}=(uv)^{-1}$, we must have $uv\in S$. Hence, in any case, $uv\in S$, and the proof is completed.
\end{proof}

\section{Idempotents of $\FSa(\Gamma)$ with norm lesser than $\frac{4}{3}$, for abelian $\Gamma$}

\noindent In this section, let $\Gamma=(\Gamma,0,+)$ be a locally compact abelian group. We aim to strengthen Saeki's Theorem \ref{Saeki} (ii) by enlarging his range of $(1,\frac{\sqrt{17} + 1}{4})$ to the optimal $(1,\frac{4}{3})$. So let $\chi_S$ be an idempotent function in $\FSa(\Gamma)$ with $\norm{\chi_S}\le \frac{4}{3}$.

Actually, in \cite{saeki2}, Saeki works with idempotents of the measure algebra $\Measures(G)$ on a locally compact abelian group $G$, and so, our $\Gamma$ and his $G$ could be considered as the Pontryagin's duals of each other. Thus $\FSa(\Gamma)\cong\Measures(G)$ isometrically, and we denote by $\mu$ the idempotent measure in $\Measures(G)$ that corresponds to $\chi_S$.

As in the previous section, we may reduce our problem to the case where $\Gamma$ is discrete. Thus suppose that $\Gamma$ is discrete, and so $G$ is compact.

Saeki's proof of Theorem \ref{Saeki} (ii) in \cite{saeki2} invokes the following lemma several times.

\begin{lemma}[Saeki]
\label{Saeki's lemma}
Assume as above. Suppose there exists $u$ and $v$ in $S$ and $w$ in $\Gamma$ such that $u+w$ belongs to $S$ but neither $v+w$ nor $v-w$ belongs to $S$. Then we have $\norm{\mu} \geq \frac{\sqrt{17} + 1}{4}$.
\end{lemma}

\noindent In the main argument, Saeki uses this lemma to show that if $S$ is not the union of two cosets of a subgroup in $\Gamma$, then $\norm{\mu} \geq \frac{\sqrt{17}+1}{4}$. The argument used breaks the problem up into many cases, and in the cases where this lemma is not used, it is always shown that in fact $\norm{\mu} \geq \frac{4}{3}$. Thus if we can strengthen this lemma, then Theorem \ref{Saeki} (ii) is strengthened also. In fact, we prove the following.

\begin{lemma}\label{stronger version of Saeki's lemma}
Assume as above. Suppose there exists $u, v \in S$, and $w \in \Gamma$, such that $u+w \in S$, and $v+w, v-w \not\in S$. Then $\norm{\chi_S}=\norm{\mu} \geq \frac{4}{3}$.
\end{lemma}

\begin{proof}
Let us define a function $f \in C(G)$ to be
\[ 
f(x) = (x,u)[2 + 2(x,w) + \frac{1}{2}(x,-w)] + (x,v)[2 - (x,w) - (x,-w)] \,.
\]
If $u-w \in S$, we get $\abs{\int f(x) \dd\mu(x)} = \frac{13}{2}$, otherwise it will simply be $6$. Next, we calculate the uniform norm $\abs{f}_G$ of $f$, by taking $x\in G$, and set $(x,w) =: e^{i\theta}$. Then we see
\begin{align*}
\abs{f(x)} &\leq \left| 2 + 2e^{i\theta} + \frac{1}{2}e^{-i\theta} \right| + \left| 2 - e^{i\theta} - e^{-i\theta} \right| \\
&= \sqrt{\frac{25}{4} + 10\cos(\theta) + 4\cos^2(\theta)} + 2 - 2\cos(\theta)= \frac{9}{2}\,.
\end{align*}
Thus $\abs{f}_G\le \frac{9}{2}$. Hence
\begin{align*}
	\norm{\mu}\ge \frac{\abs{\int_G f(x) \dd\mu(x)}}{\abs{f}_G} \ge \frac{6}{9/2}=\frac{4}{3}\,.
\end{align*}
This proves our lemma.
\end{proof}

Now we can get our desired result.

\begin{theorem}\label{stronger version of Saeki}
Let $\Gamma$ be a locally compact abelian group, and let $\varphi$ be an idempotent function in $\FSa(\Gamma)$ so that $\varphi=\chi_S$ for some nonempty $S\subseteq \Gamma$. If $\norm{\varphi} \in (1, \frac{4}{3})$, then $S$ is the union of two cosets of some open subgroup of $\Gamma$ but is not a coset itself. \enproof
\end{theorem}
\begin{proof}
This follows from the previous discussion with the argument of \cite{saeki2}.
\end{proof}

Equivalently, the above translates into the following:

\begin{theorem}\label{stronger version of Saeki}
Let $G$ be a locally compact abelian group, and let $\mu$ be an idempotent measure on $G$ with $\norm{\mu} \in (1, \frac{4}{3})$. Then 
\[
	\dd\mu(x)=\left[(-x,\gamma_1)+(-x,\gamma_2)\right]\dd m(x)\,
\]
where $m$ is the Haar measure of some compact subgroup $H$ of $G$, and $\gamma_1, \gamma_2$ are distinct characters of $H$. \enproof
\end{theorem}



\begin{thebibliography}{99}

\newcommand{\bibbook}[3]{\textsc{#1}, \emph{#2}, (#3).}
\newcommand{\bibpaper}[6]{\textsc{#1}, `#2', \emph{#3} #4 (#5) #6.}
\newcommand{\bibpreprint}[2]{\textsc{#1}, `#2', preprint.}


\bibitem{BozejkoFendler}  \bibpaper{M. Bo{\.z}ejko, and G. Fendler}
  {Herz-{S}chur multipliers and completely bounded multipliers of
              the {F}ourier algebra of a locally compact group}
	{Boll. Un. Mat. Ital. A (6)}{3}{1984}{297--302}

\bibitem{EffrosRuan}   \bibbook{E.G. Effros, and Z.-J. Ruan}
	{Operator spaces}
	{London Mathematical Society Monographs. New Series, 23, The Clarendon Press, Oxford University Press, New York}{2000}
	
\bibitem{IlieSpronk}  \bibpaper{M. Ilie, and N. Spronk}
	{Completely bounded homomorphisms of the {F}ourier algebras}
	{J. Functional Analysis}{225}{2005}{480--499}

\bibitem{levene}  \bibpaper{R. H. Levene}
	{Norms of idempotent {S}chur multipliers}
	{New York J. Math.}{20}{2014}{325--352}
	

\bibitem{Rudin}   \bibbook{W. Rudin}
	{Fourier analysis on groups}
	{Interscience Tracts in Pure and Applied Mathematics, No. 12, Interscience Publishers (a division of John Wiley and Sons),
              New York-London}{1962}	
              
\bibitem{saeki1} \bibpaper{S. Saeki}
	{On norms of idempotent measures}	
	{Proc. Amer. Math. Soc.}{19}{1968}{600--602}


\bibitem{saeki2} \bibpaper{S. Saeki}
	{On norms of idempotent measures. {II}}
	{Proc. Amer. Math. Soc.}{19}{1968}{367--371}
	
\bibitem{Spronk} \bibpaper{N. Spronk}
	{Measurable {S}chur multipliers and completely bounded
              multipliers of the {F}ourier algebras}
	{Proc. London Math. Soc. (3)}{89}{2004}{161--192}

\bibitem{stan}  \bibpaper{A.-M. P. Stan}
	{On idempotents of completely bounded multipliers of the
              {F}ourier algebra {$A(G)$}}
           {Indiana Univ. Math. J.}{58}{2009}{523--535}
           
\bigskip

\end{thebibliography}
\end{document}